\documentclass[12pt, reqno]{amsart}
\usepackage{amsmath, amsthm, amscd, amsfonts, amssymb, graphicx, color}
\usepackage[bookmarksnumbered, colorlinks, plainpages]{hyperref}
\hypersetup{colorlinks=true,linkcolor=red, anchorcolor=green, citecolor=cyan, urlcolor=red, filecolor=magenta, pdftoolbar=true}

\newtheorem{theorem}{Theorem}[section]
\newtheorem{lemma}[theorem]{Lemma}

\newtheorem{cor}[theorem]{Corollary}
\theoremstyle{definition}

\theoremstyle{remark}
\newtheorem{remark}[theorem]{\bf{Remark}}
\numberwithin{equation}{section}
\begin{document}

\title {On inequalities for A-numerical radius of operators} 

\author{Pintu Bhunia, Kallol Paul and Raj Kumar Nayak}

\address{(Bhunia) Department of Mathematics, Jadavpur University, Kolkata 700032, West Bengal, India.}
\email{pintubhunia5206@gmail.com}

\address{(Paul) Department of Mathematics, Jadavpur University, Kolkata 700032, West Bengal, India.}
\email{kalloldada@gmail.com}

\address{(Nayak) Department of Mathematics, Jadavpur University, Kolkata 700032, West Bengal, India.}
\email{rajkumarju51@gmail.com}




\maketitle
\begin{abstract}
Let  $A$ be a positive operator on a complex Hilbert space $\mathcal{H}.$  We present inequalities concerning upper and lower bounds for $A$-numerical radius of operators, which improve on and generalize the existing ones, studied recently in [A. Zamani, A-Numerical radius inequalities for semi-Hilbertian space operators, Linear Algebra Appl. 578 (2019) 159-183].  We also obtain some inequalities for $B$-numerical radius of $2\times 2$ operator matrices where $B$ is the $2\times 2$ diagonal operator matrix whose diagonal entries are $A$. Further we obtain upper bounds for $A$-numerical radius for product of operators which improve on the existing bounds.  
\end{abstract}

\section{Introduction}
\noindent Let $\mathcal{H}$ be a complex Hilbert space with usual inner product $\langle .,. \rangle$ and $\|.\|$ be the norm induced from $\langle .,. \rangle$. Let $\mathcal{B}(\mathcal{H})$ denote the $C^*$-algebra of all bounded linear operators on $\mathcal{H}.$ Throughout this article we assume $I$ and $O$ are identity operator and zero operator on $\mathcal{H}$, respectively. A self-adjoint operator $A\in \mathcal{B}(\mathcal{H})$ is called positive  if $\langle Ax,x \rangle \geq 0$ for all $x\in \mathcal{H}$ and  is called strictly positive if $\langle Ax,x \rangle > 0$ for all $(0\neq)x\in \mathcal{H}$. For a positive (strictly positive) operator $A$ we write $A\geq 0$ $(A>0)$. Let $B=\left(\begin{array}{cc}
	A&O \\
	O&A
\end{array}\right)$. Then $B \in \mathcal{B}(\mathcal{H} \oplus \mathcal{H})$ is positive or strictly positive if $A$ is positive or strictly positive respectively. Let us fix the alphabets $A$ and  $B$ for positive operator on $\mathcal{H}$ and $ \mathcal{H} \oplus \mathcal{H}$ respectively. Clearly $A$ induces a positive semidefinite sesquilinear form $\langle .,. \rangle _A : \mathcal{H} \times \mathcal{H} \rightarrow \mathbb{C}$ defined as $\langle x,y \rangle _A=\langle Ax,y \rangle$ for $x,y\in \mathcal{H}$. Let $\|.\|_A$ denote the semi-norm on $\mathcal{H}$ induced from the  sesquilinear form $\langle .,. \rangle_A,$ that is, $\|x\|_A=\sqrt{\langle x,x \rangle_A}$ for all $x\in \mathcal{H}.$ It is easy to verify that $\|.\|_A$ is a norm if and only if $A$ is a strictly positive operator. Also $(\mathcal{H}, \|.\|_A)$ is complete space if and only if the range $\mathcal{R}(A)$ of $A$ is closed in $\mathcal{H}$. By $\overline {\mathcal{R}(T)}$ we denote the norm closure of $\mathcal{R}(T)$ in $\mathcal{H}$. For $T\in \mathcal{B}(\mathcal{H})$, A-operator semi-norm of $T$, denoted as $\|T\|_A$,  is defined as \[\|T\|_A=\sup_{x\in  \overline{\mathcal{R}(A)},x\neq 0}\frac{\|Tx\|_A}{\|x\|_A}.\] 
Here we note that for a given $T \in  \mathcal{B}(\mathcal{H})$, if there exists $c>0$ such that $\|Tx\|_A \leq c\|x\|_A$ for all $x\in \overline{\mathcal{R}(A)}$ then $\|T\|_A<+\infty$. 
Again A-minimum modulus of $T$, denoted as $m_A(T)$ (see \cite{Z2}), is defined as \[m_A(T)=\inf_{x\in  \overline{\mathcal{R}(A)},x\neq 0}\frac{\|Tx\|_A}{\|x\|_A}.\]
We set $\mathcal{B}^A(\mathcal{H})=\{ T\in \mathcal{B}(\mathcal{H}): \|T\|_A <+ \infty\}.$ It is easy to verify that $\mathcal{B}^A(\mathcal{H})$ is not generally a subalgebra of $\mathcal{B}(\mathcal{H})$ and $\|T\|_A=0$ iff $ATA=0.$ 
For $T \in \mathcal{B}(\mathcal{H})$, an operator $R \in \mathcal{B}(\mathcal{H})$ is called an $A$-adjoint of $T$ if for every $x,y \in \mathcal{H}$ such that $\langle Tx,y \rangle_A=\langle x,Ry \rangle_A$, that is, $AR=T^*A$ where $T^*$ is the adjoint of $T$. For any operator $T\in \mathcal{B}(\mathcal{H})$, A-adjoint of $T$ may or may  not exist. In fact, an operator $T \in \mathcal{B}(\mathcal{H})$ may have one or more than one A-adjoint operators, also it may have none. By Douglas Theorem \cite{doug}, we have  that an operator $T\in \mathcal{B}(\mathcal{H})$ admits A-adjoint if 
  $$\mathcal{R}(T^{*}A)\subseteq \mathcal{R}(A). $$ 
Now we consider an example that  $A=\left(\begin{array}{cc}
0&0\\
0&1
\end{array}\right)$ and $T=\left(\begin{array}{cc}
0&1\\
1&0
\end{array}\right)$ on $\mathbb{C}^2.$ Then we see that $\mathcal{R}(T^{*}A)=\{(x,0):x\in \mathbb{C}\}$ and $\mathcal{R}(A)=\{(0,x):x\in \mathbb{C}\}.$ So by Douglas Theorem \cite{doug} we conclude that $T$  have no A-adjoint. 
\noindent Let $\mathcal{B}_A(\mathcal{H})$ be the collection of all operators in $ \mathcal{B}^A(\mathcal{H})$ which admits A-adjoint. Note that $\mathcal{B}_A(\mathcal{H})$ is a sub-algebra of $\mathcal{B}(\mathcal{H})$ which is neither closed nor dense in $\mathcal{B}(\mathcal{H}).$  For $T \in \mathcal{B}(\mathcal{H})$, A-adjoint operator of $T$ is written as $T^{\sharp_A}$. It is well known that $T^{\sharp_A}=A^\dag T^*A$ where $A^\dag$ is the Moore-Penrose inverse of $A$, (see \cite{MKX}). It is useful that if $T \in \mathcal{B}_A(\mathcal{H})$ then $AT^{\sharp_A}=T^*A$. An operator $T\in \mathcal{B}_A(\mathcal{H})$ is said to be A-self-adjoint operator if $AT$ is self-adjoint, that is, $AT=T^*A$ and it is called A-positive if $AT\geq 0$. For A-positive operator $T$ we have 
\[\|T\|_A=\sup\{\langle Tx,x \rangle_A: x\in \mathcal{H}, \|x\|_A=1\}.\] 
An operator $U\in  \mathcal{B}_A(\mathcal{H})$ is said to be $A$-unitary if $U^{\sharp_A}U=(U^{\sharp_A})^{\sharp_A}U^{\sharp_A}=P_A$, $P_A$ is the orthogonal projection onto $\overline{\mathcal{R}(A)}.$  Here we note that if $T\in \mathcal{B}_A(\mathcal{H})$ then $T^{\sharp_A}\in \mathcal{B}_A(\mathcal{H})$, $(T^{\sharp_A})^{\sharp_A}=P_ATP_A$. Also $T^{\sharp_A}T$, $TT^{\sharp_A}$ are A-self-adjoint and A-positive operators and so 
\[\|T^{\sharp_A}T\|_A =\|TT^{\sharp_A}\|_A =\|T\|^2_A =\|T^{\sharp_A}\|^2_A.\]
Also for $T,S \in  \mathcal{B}_A(\mathcal{H})$,  $(TS)^{\sharp_A}=S^{\sharp_A}T^{\sharp_A}$, $\|TS\|_A\leq \|T\|_A\|S\|_A$ and $\|Tx\|_A\leq \|T\|_A\|x\|_A$ for all $x\in \mathcal{H}.$ For further details we refer the reader to \cite{ACG, ACG2, AS}. For an operator $T\in \mathcal{B}_A(\mathcal{H})$, we write $\textit{Re}_A(T)=\frac{1}{2}(T+T^{\sharp_A})$ and $\textit{Im}_A(T)=\frac{1}{2i}(T-T^{\sharp_A})$.\\
For $T\in \mathcal{B}_A(\mathcal{H})$, A-numerical radius of $T$, denoted as $w_A(T)$, is defined  as  
\[ w_A(T)=\sup\{|\langle Tx,x \rangle_A|: x \in \mathcal{H}, \|x\|_A=1\},  ~~\mbox{(see \cite{BFA})}. \]
Also, for $T\in \mathcal{B}_A(\mathcal{H})$, A-Crawford number of $T$, denoted as $c_A(T)$ (see  \cite{Z2}), is defined as  
\[c_A(T)=\inf\{|\langle Tx,x \rangle_A|: x\in \mathcal{H}, \|x\|_A=1\}.\]
For $T\in \mathcal{B}_A(\mathcal{H})$, it is well-known that A-numerical radius of $T$ is equivalent to A-operator semi-norm of $T$, (see \cite{Z}), satisfying the following inequality: 
\[\frac{1}{2}\|T\|_A \leq w_A(T)\leq \|T\|_A.\]

\noindent Over the years many mathematicians have studied numerical radius inequalities in \cite{BBP1,B,  BBP, BBP2, D, GR, HKS, HKS2, KMY, K, PB, PB2, S, Y}. 
Recently,  Zamani  \cite{Z}  have studied A-numerical radius and computed some inequalities for A-numerical radius. In this paper, we compute some inequalities for B-numerical radius of $2 \times 2$ operator matrices which generalize and improve on the existing inequalities. Also we obtain some inequalities for A-numerical radius of operators in  $\mathcal{B}_A(\mathcal{H})$ which improve on the existing inequalities in \cite{Z}. Further we obtain  A-numerical radius bounds for  sum of product of operators in $\mathcal{B}_A(\mathcal{H})$ which improve on the existing bounds.

\section{\textbf{A-numerical radius inequalities for operators in $\mathcal{B}_A(\mathcal{H})$}}
We begin this section with the following three results proved by Zamani \cite{Z}.

\begin{lemma}\label{lemma:Z3}
	Let $T\in \mathcal{B}_A(\mathcal{H})$ be an $A$-self-adjoint operator. Then
	\[w_A(T)=\|T\|_A.\]
\end{lemma}

\begin{lemma}\label{lemma:Z1}
Let $T\in \mathcal{B}_A(\mathcal{H})$. For every $\theta \in \mathbb{R},$
\[w_A\left (\textit{Re}_A(e^{i\theta }T) \right )=\left\| \textit{Re}_A(e^{i\theta }T) \right\|_A.\]
\end{lemma}

\begin{lemma}\label{lemma:Z2}
Let $T\in \mathcal{B}_A(\mathcal{H})$. Then
\[w_A(T)=\sup_{\theta \in \mathbb{R}}\left\|\textit{Re}_A(e^{i\theta }T)\right\|_A ~~\mbox{ and } ~~w_A(T)=\sup_{\theta \in \mathbb{R}}\left\|\textit{Im}_A(e^{i\theta }T)\right\|_A. \]
\end{lemma}

\noindent Next we compute B-numerical radius for some $2 \times 2$ operator matrices. First we note that the operator $T=(T_{ij})_{2 \times 2}$ is in $\mathcal{B}_B(\mathcal{H} \oplus \mathcal{H})$ if the operator $T_{ij}$  (for $i,j=1,2$) are in $\mathcal{B}_A(\mathcal{H})$ and in this case (see \cite[Lemma 3.1]{BFP}) $T^{\sharp_B}=(T_{ji}^{\sharp_A})_{2 \times 2}.$ We now prove the following lemma.  
  
\begin{lemma}\label{lemma:1}
Let $X,Y \in \mathcal{B}_A(\mathcal{H}) $. Then the following results hold:

\begin{eqnarray*}
&&(i) ~~ w_B \left(\begin{array}{cc}
X&O \\
O&Y
\end{array}\right)=\max \left \{ w_A(X), w_A(Y)\right \}.\\
&&(ii)~~ \textit{If}~~ A>0 ~~\textit{then}~~  w_B \left(\begin{array}{cc}
O&X \\
Y&O
\end{array}\right) =  w_B\left(\begin{array}{cc}
O&Y \\
X&O
\end{array}\right). \\
&&(iii) ~~ \textit{If} ~~ A>0~~ \textit{then}~~  \textit{for any}~~ \theta \in \mathbb{R},~~  
	w_B\left(\begin{array}{cc}
	O&X \\
	e^{i\theta}Y&O
	\end{array}\right) = w_B\left(\begin{array}{cc}
	O&X \\
	Y&O
	\end{array}\right).\\
&&(iv)~~ \textit{If} ~~ A>0~~ \textit{then}~~
w_B \left(\begin{array}{cc}
X&Y \\
Y&X
\end{array}\right)=\max \left \{ w_A(X+Y), w_A(X-Y)\right \}.\\ 
&& \textit{In particular}, ~~w_B \left(\begin{array}{cc}
O&Y \\
Y&O
\end{array}\right)=w_A(Y).
\end{eqnarray*}
\end{lemma}

\begin{proof}
(i) Let $T=\left(\begin{array}{cc}
X&O \\
O&Y
\end{array}\right)$ and $u=(x,y)\in \mathcal{H} \oplus \mathcal{H}$ with $\|u\|_B=1$, i.e., $\|x\|^2_A+\|y\|^2_A=1.$ Now,
\begin{eqnarray*}
	|\langle Tu,u \rangle_B| &\leq& |\langle Xx,x \rangle_A|+|\langle Yy,y \rangle_A|\\
	&\leq& w_A(X)\|x\|^2_A + w_A(Y)\|y\|^2_A \\
	&\leq& \max \left \{ w_A(X), w_A(Y)\right \}.
\end{eqnarray*} 
Taking supremum over $\|u\|_B =1$, we get 
\[w_B(T) \leq  \max \left \{ w_A(X), w_A(Y)\right \}.\]
Suppose $u = (x,0) \in  \mathcal{H} \oplus \mathcal{H} $ where $\|x\|_A =1.$ Then 
\[|\langle Tu,u\rangle_B|=|\langle AXx,x\rangle|=|\langle Xx,x\rangle_A|.\]
Taking supremum over $\|x\|_A=1$, we get
\[ \sup_{\|x\|_A=1}|\langle Tu,u\rangle_B|=w_A(X)\]
and so we have $w_B(T) \geq w_A(X).$
Similarly, if we take $v =(0,y) \in \mathcal{H} \oplus \mathcal{H}$ with $\|y\|_A=1$ then we can show that  $w_B(T) \geq w_A(Y).$ Therefore, $w_B(T) \geq \max \left \{ w_A(X), w_A(Y)\right \}.$ This completes the proof of Lemma \ref{lemma:1} (i).\\

\noindent(ii) The proof follows from the observation that $w_B(U^{\sharp_B}TU) = w_B(T)$ (see \cite[Lemma 3.8]{BFP}) if  $U$ is an $B$-unitary operator on $\mathcal{H} \oplus \mathcal{H}$, here we take $U=\left(\begin{array}{cc}
O&I \\
I&O
\end{array}\right)$.\\

\noindent(iii) As in (ii) we now take $U=\left(\begin{array}{cc}
I&O \\
O&e^\frac{i\theta}{2}I
\end{array}\right)$.\\

\noindent(iv) Let  $U= \frac{1}{\sqrt{2}}\left(\begin{array}{cc}
I&I \\
-I&I
\end{array}\right)$ and $T= \left(\begin{array}{cc}
X&Y \\
Y&X
\end{array}\right)$. Then an easy calculation we have
\[U^{\sharp_B}TU =  \left(\begin{array}{cc}
X-Y&O \\
O&X+Y
\end{array}\right).\] Using  Lemma \ref{lemma:1} (i) and $w_B(U^{\sharp_B}TU) = w_B(T)$ we get 
\[w_B(T)=\max \left \{ w_A(X+Y), w_A(X-Y)\right \}.\]
Taking $X=O$ we get 
\[w_B \left(\begin{array}{cc}
O&Y \\
Y&O
\end{array}\right)=w_A(Y).\] 
This completes the proof of Lemma \ref{lemma:1} (iv).
\end{proof}

Next we prove the following  important lemma for $A$-positive operators.

\begin{lemma}\label{lemma:5}
Let $X, Y \in \mathcal{B}_A(\mathcal{H})$ be A-positive. If $X-Y$ is A-positive then \[ \|X\|_A \geq \|Y\|_A.\]
\end{lemma}
\begin{proof}
From the definition of A-positive operator we have , for all $x\in \mathcal{H}$
\begin{eqnarray*}
	&&\langle (X-Y)x,x\rangle_A \geq0\\
	&\Rightarrow & \langle Xx,x\rangle_A  \geq \langle Yx,x\rangle_A\\
	&\Rightarrow & w_A(X) \geq \langle Yx,x\rangle_A. 
\end{eqnarray*}
Taking supremum over $\|x\|_A=1,$ we get
\[w_A(X) \geq w_A(Y).\]
Since $X,Y$ are A-self-adjoint operators, so $ \|X\|_A \geq \|Y\|_A$.
\end{proof}

We are now in a position to prove the following theorem.

\begin{theorem} \label{theorem:1}
Let $X,Y \in \mathcal{B}_A(\mathcal{H})$. Then
\begin{eqnarray*}
	w^{2}_B\left(\begin{array}{cc}
		O&X \\
		Y&O
	\end{array}\right)  &\geq& \frac{1}{4}\max \big\{\| XX^{\sharp_A}+Y^{\sharp_A}Y\|_A, \| X^{\sharp_A}X+YY^{\sharp_A}\|_A\big\} ~~\mbox{and} \\ 
	w^{2}_B\left(\begin{array}{cc}
		O&X \\
		Y&O
	\end{array}\right) 
	&\leq& \frac{1}{2}\max \big\{\| XX^{\sharp_A}+Y^{\sharp_A}Y\|_A, \| X^{\sharp_A}X+YY^{\sharp_A}\|_A\big\}. 
\end{eqnarray*}
\end{theorem}

\begin{proof}
 Let $T=\left(\begin{array}{cc}
O&X \\
Y&O
\end{array}\right)$, $H_{\theta}=\textit{Re}_A(e^{i\theta}T)$ and $K_{\theta}=\textit{Im}_A(e^{i\theta}T).$ 
Then from an easy calculation we have, 
\[H^2_{\theta}+K^2_{\theta}=\frac{1}{2}\left(\begin{array}{cc}
M&O \\
O&N
\end{array}\right)\] where $M=XX^{\sharp_A}+Y^{\sharp_A}Y$, $N=X^{\sharp_A}X+YY^{\sharp_A}$.\\
Taking norm on both sides and then using Lemma \ref{lemma:Z2}, we get 
 \[\frac{1}{2}\left\|\left(\begin{array}{cc}
M&O \\
O&N
\end{array}\right)\right\|_B=\|H^2_{\theta}+K^2_{\theta}\|_B\leq \|H_{\theta}\|^2_B+\|K_{\theta}\|^2_B\leq 2w^2_B(T).\]
Therefore we get, 
\[\frac{1}{2}\max \big\{\| M\|_A, \| N\|_A \big\}\leq 2w^2_B(T).\]
This completes the proof of the first inequality.\\
Again, from $H^2_{\theta}+K^2_{\theta}=\frac{1}{2}\left(\begin{array}{cc}
M&O \\
O&N
\end{array}\right) $~~\mbox{we have,}  $H^2_{\theta}-\frac{1}{2}\left(\begin{array}{cc}
M&O \\
O&N
\end{array}\right)=-K^2_{\theta}\leq 0.$
Therefore, $ H^2_{\theta}\leq\frac{1}{2}\left(\begin{array}{cc}
M&O \\
O&N
\end{array}\right)$. Using Lemma \ref{lemma:5}, we get  
\[\|H_{\theta}\|^2_B\leq\frac{1}{2}\left\|\left(\begin{array}{cc}
M&O \\
O&N
\end{array}\right)\right\|_B=\frac{1}{2}\max \big\{\| M\|_A, \|N\|_A\big\}.\]
Taking supremum over $\theta\in \mathbb{R}$, we get 
\[w^{2}_B(T) \leq \frac{1}{2}\max \big\{\| M\|_A, \|N\|_A\big\}.\] This completes the proof of the second inequality of the theorem. 
\end{proof}

Next we state the corollary, the proof of which follows easily by considering $X=Y=T$ and $A>0$ in Theorem \ref{theorem:1}.
\begin{cor}\label{corollary}
Let $T\in \mathcal{B}_A(\mathcal{H})$ and $A>0.$	Then 
\[\frac{1}{4}\| TT^{\sharp_A}+T^{\sharp_A}T\|_A \leq w^2_A(T) \leq \frac{1}{2}\| TT^{\sharp_A}+T^{\sharp_A}T\|_A.\]
\end{cor}

\begin{remark}\label{remark:1}
(i) Kittaneh \cite[Th. 1]{K} proved that if $T\in \mathcal{B}(\mathcal{H})$ then	
\[\frac{1}{4}\| TT^{*}+T^{*}T\| \leq w^2(T) \leq \frac{1}{2}\| TT^{*}+T^{*}T\|,\]	
which follows easily from  Corollary \ref{corollary} by taking $A=I.$\\
(ii)  Zamani \cite[Th. $2.10$]{Z} proved that 	
	\[  w^2_A(T) \leq \frac{1}{2}\| TT^{\sharp_A}+T^{\sharp_A}T\|_A,\]
which clearly follows from the   inequality obtained in Corollary \ref{corollary}.  
\end{remark}

Next we prove the following theorem. 
  
\begin{theorem} \label{theorem:2}
Let $X, Y \in \mathcal{B}_A(\mathcal{H})$. Then\\
$w^{4}_B\left(\begin{array}{cc}
O&X \\
Y&O
\end{array}\right)  \geq  \frac{1}{16}\max \big\{\| P\|_A, \| Q\|_A\big\} ~\mbox{and}$ 
\begin{eqnarray*}
w^{4}_B\left(\begin{array}{cc}
O&X \\
Y&O
\end{array}\right) \leq \frac{1}{8}\max \Big\{\| XX^{\sharp_A}+Y^{\sharp_A}Y\|^2_A+4w^2_A(XY),\\
~~~~ \| X^{\sharp_A}X+YY^{\sharp_A}\|^2_A  + 4w^2_A(YX)\Big\},
\end{eqnarray*}
where $P=(XX^{\sharp_A}+Y^{\sharp_A}Y)^2+4(\textit{Re}_A(XY))^2,Q=(X^{\sharp_A}X+YY^{\sharp_A})^2+4(\textit{Re}_A(YX))^2.$
\end{theorem}

\begin{proof} Let $T=\left(\begin{array}{cc}
O&X \\
Y&O
\end{array}\right)$, $H_{\theta}=\textit{Re}_A(e^{i\theta}T)$ and $K_{\theta}=\textit{Im}_A(e^{i\theta}T).$
 Then we get,
\[H^4_{\theta}+K^4_{\theta}=\frac{1}{8}\left(\begin{array}{cc}
P_0&O \\
O&Q_0
\end{array}\right).\]
where $P_0=(XX^{\sharp_A}+Y^{\sharp_A}Y)^2+4(\textit{Re}_A(e^{2i\theta}XY))^2$, $Q_0=(X^{\sharp_A}X+YY^{\sharp_A})^2+4(\textit{Re}_A(e^{2i\theta}YX))^2.$
Taking norm on both sides and using Lemma \ref{lemma:Z2}, we get \\ $\frac{1}{8} \left \|\left(\begin{array}{cc}
P_0&O \\
O&Q_0
\end{array}\right)\right \|_B=\|H^4_{\theta}+K^4_{\theta}\|_B\leq \|H_{\theta}\|^4_B+\|K_{\theta}\|^4_B\leq 2w^4_B(T).$\\  Therefore we get,
\[\frac{1}{8}\max \big\{\| P_0\|_A, \| Q_0\|_A\big\}\leq 2w^4_B(T).\] 
This holds for all $\theta \in \mathbb{R}$, so taking $\theta=0$ we get,\\
$\frac{1}{8}\max \big\{\| P\|_A, \| Q\|_A\big\} \leq 2w^{4}_B(T).$ 
This completes the proof of the first inequality of the theorem.

\noindent Again, from $H^4_{\theta}+K^4_{\theta}=\frac{1}{8}\left(\begin{array}{cc}
P_0&O \\
O&Q_0
\end{array}\right)$ we have,
$H^4_{\theta}-\frac{1}{8}\left(\begin{array}{cc}
P_0&O \\
O&Q_0
\end{array}\right)=-K^4_{\theta}\leq 0.$
Therefore, $ H^4_{\theta}\leq\frac{1}{8}\left(\begin{array}{cc}
P_0&O \\
O&Q_0
\end{array}\right).$\\
Using Lemma \ref{lemma:5}, we get
\[\|H_{\theta}\|_B^4\leq\frac{1}{8}\left\|\left(\begin{array}{cc}
P_0&O \\
O&Q_0
\end{array}\right)\right\|_B=\frac{1}{8}\max \big\{\| P_0\|_A, \| Q_0\|_A\big\}.\]
Therefore using Lemma \ref{lemma:Z2}, we get
\[\|H_{\theta}\|_B^4\leq\frac{1}{8}\max  \big\{\| XX^{\sharp_A}+Y^{\sharp_A}Y\|^2_A+4w^2_A(XY), \| X^{\sharp_A}X+YY^{\sharp_A}\|^2_A+4w^2_A(YX)\big\}.\] 
Taking supremum over $\theta\in \mathbb{R}$ and using Lemma \ref{lemma:Z2}, we get 
\[w^{4}_B(T) \leq \frac{1}{8}\max  \big\{\| XX^{\sharp_A}+Y^{\sharp_A}Y\|^2_A+4w^2_A(XY), \| X^{\sharp_A}X+YY^{\sharp_A}\|^2_A+4w^2_A(YX)\big\}.\] This completes the proof of the second inequality of the theorem. 
\end{proof}

Now, taking $X=Y=T$ (say) and $A>0$ in the above Theorem  \ref{theorem:2}, we get the following inequality. 
\begin{cor}\label{cor1}
Let $T\in B_A(H)$ where $A>0$. Then
\begin{eqnarray*}
	\frac{1}{16}\|(TT^{\sharp_A}+T^{\sharp_A}T)^2+4(\textit{Re}_A(T^2))^2\|_A &\leq & w^4_A(T)\\
	&\leq & \frac{1}{8}\| TT^{\sharp_A}+T^{\sharp_A}T\|^2_A+\frac{1}{2}w^2_A(T^2). 
\end{eqnarray*}
\end{cor}

\begin{remark}
(i) In \cite[Th. 2.11]{BBP1}	we proved that if $T\in \mathcal{B}(\mathcal{H})$ then
\begin{eqnarray*}
	\frac{1}{16}\|TT^{*}+T^{*}T\|^2+\frac{1}{4}m\left((\textit{Re}(T^2))^2\right) &\leq & w^4(T)\\
	&\leq & \frac{1}{8}\| TT^{*}+T^{*}T\|^2+\frac{1}{2}w^2(T^2), 
\end{eqnarray*}
which follows easily from Corollary \ref{cor1} by taking $A=I.$\\
(ii) Zamani \cite[Th. $2.10$]{Z} proved that 	
\[  w^2_A(T) \leq \frac{1}{2}\| TT^{\sharp_A}+T^{\sharp_A}T\|_A.\]
Since $w_A(T^2)\leq w^2_A(T)$ (see \cite[Prop. 3.10]{MXZ}), so $w_A(T^2)\leq \frac{1}{2}\|TT^{\sharp_A}+T^{\sharp_A}T\|_A.$ Therefore, the right hand inequality obtained in Corollary \ref{cor1} improves on  the inequality obtained by Zamani \cite[Th. $2.10$]{Z}.
\end{remark}

We next prove the following theorem.

\begin{theorem}\label{theorem:3}
	Let $T \in \mathcal{B}_A(\mathcal{H})$ where $A>0$. Then  
	\[w^{4}_A(T)\leq \frac{1}{4}w^2_A(T^2)+\frac{1}{8}w_A(T^2P+PT^2)+\frac{1}{16}\|P\|_A^2,\]  where $P=T^{\sharp_A}T+TT^{\sharp_A}.$
\end{theorem}

\begin{proof}
	From Lemma \ref{lemma:Z2}, we have $w_A(T)=\sup_{\theta \in \mathbb{R}}\|H_\theta\|_A$ where $H_{\theta}=\textit{Re}_A(e^{i\theta}T)$. Then
	\begin{eqnarray*}
		H_{\theta} & = & \frac{1}{2}(e^{i\theta} T + e^{-i\theta} T^{\sharp_A}) \\
		\Rightarrow 4 {H_{\theta}}^2  &= & e^{2i\theta} T^2 + e^{-2i\theta} {T^{\sharp_A}}^2 + P \\ 
		\Rightarrow 16 {H_{\theta}}^4  &= & \big(e^{2i\theta} T^2 + e^{-2i\theta} {T^{\sharp_A}}^2 + P \big) \big(e^{2i\theta} T^2 + e^{-2i\theta} {T^{\sharp_A}}^2 + P \big)\\
		&= & \big(e^{2i\theta} T^2 + e^{-2i\theta} {T^{\sharp_A}}^2 \big)^2+\big(e^{2i\theta} T^2 + e^{-2i\theta} {T^{\sharp_A}}^2 \big)P \\ 
		&& +P \big(e^{2i\theta} T^2 + e^{-2i\theta} {T^{\sharp_A}}^2 \big)+P^2\\
		& = & 4\big(\textit{Re}_A(e^{2i\theta} T^2)\big)^2+ 2 \textit{Re}_A(e^{2i\theta} (T^2P+PT^2))  +P^2\\	
		\Rightarrow \| {H_{\theta}}^4 \|_A &\leq &	\frac{1}{4}\big \|	\textit{Re}_A(e^{2i\theta} T^2)	\big \|_A^2+\frac{1}{8}\big \|	\textit{Re}_A(e^{2i\theta} (T^2P+PT^2))	\big \|_A +\frac{1}{16} \|P\|_A^2\\
		&\leq & \frac{1}{4}w^2_A(T^2)+\frac{1}{8}w_A(T^2P+PT^2)+\frac{1}{16}\|P\|_A^2.
	\end{eqnarray*}
	Taking supremum over $\theta \in \mathbb{R}$, we get,
	\begin{eqnarray*}
		\Rightarrow w^{4}_A(T)&\leq&\frac{1}{4}w^2_A(T^2)+\frac{1}{8}w_A(T^2P+PT^2)+\frac{1}{16}\|P\|_A^2.
	\end{eqnarray*}
\end{proof}

\begin{remark}
	Using the inequality in Corollary \ref{cor2}, it is easy to see that if $A>0$ then $w_A(T^2P+PT^2)\leq 2w_A(T^2)\|P\|_A.$ In case $A>0$,  we would like to remark that the inequality obtained in Theorem \ref{theorem:3} improves on the inequality \cite[Th. 2.11]{Z} obtained by Zamani. As for numerical example, if we consider $T=\left(\begin{array}{ccc}
	0&1&0 \\
	0&0&2\\
	0&0&0
	\end{array}\right)$ and $A=\left(\begin{array}{ccc}
	1&0&0 \\
	0&1&0\\
	0&0&1
	\end{array}\right)$ on $\mathbb{C}^3$, then by simple computation we have
	\[\frac{1}{4}w^2_A(T^2)+\frac{1}{8}w_A(T^2P+PT^2)+\frac{1}{16}\|P\|_A^2=\frac{39}{16} < \frac{1}{16}\left (\| P\|_A+2w_A(T^2)\right )^2=\frac{49}{16}.\]
\end{remark}

Now we prove the following theorem.

\begin{theorem}\label{theorem:4}
	Let $T \in \mathcal{B}_A(\mathcal{H})$ where $A>0$. Then 
	\[w^{3}_A(T)\leq \frac{1}{4}w_A(T^3)+\frac{1}{4}w_A(T^2T^{\sharp_A}+{T^{\sharp_A}}T^2+TT^{\sharp_A}T).\] Moreover if $T^2=0$ then $w_A(T)=\frac{1}{2}\sqrt{\|TT^{\sharp_A}+{T^{\sharp_A}}T\|_A}$ and if $T^3=0$ then $w^3_A(T)=\frac{1}{4}w_A(T^2T{^{\sharp_A}}+{T{^{\sharp_A}}}T^2+TT{^{\sharp_A}}T)$.
\end{theorem}

\begin{proof}
From Lemma \ref{lemma:Z2}, we have $w_A(T)=\sup_{\theta \in \mathbb{R}}\|H_\theta\|_A$ where $H_{\theta}=\textit{Re}_A(e^{i\theta}T)$. Then,
	\begin{eqnarray*}
		H_{\theta} & = & \frac{1}{2}(e^{i\theta} T + e^{-i\theta} T^{\sharp_A}) \\
		\Rightarrow 4 {H_{\theta}}^2  &= & e^{2i\theta} T^2 + e^{-2i\theta} {T^{\sharp_A}}^2 + T^{{\sharp_A}}T+TT^{{\sharp_A}} \\ 
		\Rightarrow 8H^3_{\theta}  &= &\big( e^{2i\theta} T^2 + e^{-2i\theta} {T^{\sharp_A}}^2 + T^{{\sharp_A}}T+TT^{{\sharp_A}}\big)(e^{i\theta} T + e^{-i\theta} T^{\sharp_A})\\
		\Rightarrow H^3_{\theta}  &= & \frac{1}{4}\textit{Re}_A(e^{3i\theta} T^3) +\frac{1}{4}\textit{Re}_A(e^{i\theta} (T^2T{^{\sharp_A}}+{T{^{\sharp_A}}}T^2+TT{^{\sharp_A}}T)\\
		\Rightarrow \|H^3_{\theta}\|_A  &\leq & \frac{1}{4}\|\textit{Re}_A(e^{3i\theta} T^3)\|_A +\frac{1}{4}\|\textit{Re}_A(e^{i\theta} (T^2T{^{\sharp_A}}+{T{^{\sharp_A}}}T^2+TT{^{\sharp_A}}T))\|_A\\
		&\leq & \frac{1}{4}w_A(T^3)+\frac{1}{4}w_A(T^2T^{\sharp_A}+{T^{\sharp_A}}T^2+TT^{\sharp_A}T). 
	\end{eqnarray*}
	Taking supremum over $\theta \in \mathbb{R}$, we get the desired inequality.\\
	If $T^2=0$, then $4 {H_{\theta}}^2  =  T^{\sharp_A}T+TT^{\sharp_A}$ and so $w_A(T)=\frac{1}{2}\sqrt{\|TT^{\sharp_A}+{T^{\sharp_A}}T\|_A}.$\\
	If $T^3=0$, then $H^3_{\theta} = \frac{1}{4}\textit{Re}_A(e^{i\theta} (T^2T{^{\sharp_A}}+{T{^{\sharp_A}}}T^2+TT{^{\sharp_A}}T))$ and so $w^3_A(T) = \frac{1}{4}w_A(T^2T{^{\sharp_A}}+{T{^{\sharp_A}}}T^2+TT{^{\sharp_A}}T)$.
\end{proof}

\begin{remark}
Here we would like to remark that the bound obtained in  Theorem \ref{theorem:4} improves on the existing upper bound in \cite[Cor. 2.8]{Z} when $A>0$. Note that  if $T^2=0$ then $w_A(T)=\frac{1}{2}\sqrt{\|TT^{\sharp_A}+{T^{\sharp_A}}T\|_A}$. But converse is not true, that is, $w_A(T)=\frac{1}{2}\sqrt{\|TT^{\sharp_A}+{T^{\sharp_A}}T\|_A}$ does not always imply $T^2=O.$ As for example we consider $T=\left(\begin{array}{ccc}
	0&2&0 \\
	0&0&0\\
	0&0&1
	\end{array}\right)$ and $A=\left(\begin{array}{ccc}
	1&0&0 \\
	0&1&0\\
	0&0&1
	\end{array}\right)$ on $\mathbb{C}^3$. Then we see that  $w_A(T)=\frac{1}{2}\sqrt{\|TT^{\sharp_A}+{T^{\sharp_A}}T\|_A}=1$ but $T^2=\left(\begin{array}{ccc}
	0&0&0 \\
	0&0&0\\
	0&0&1
	\end{array}\right)\neq O.$
\end{remark}

Next we prove the following inequality.

\begin{theorem}\label{theorem:5}
Let $T \in \mathcal{B}_A(\mathcal{H})$. Then for each $r\geq 1$, 
	\[w^{2r}_A(T)\leq \frac{1}{2}w^r_A(T^2)+\frac{1}{4}\big\|(T^{\sharp_A}T)^r+(TT^{\sharp_A})^r\big\|_A.\] 
\end{theorem}
\begin{proof}
From Lemma \ref{lemma:Z2}, we get $w_A(T)=\sup_{\theta \in \mathbb{R}}\|H_\theta\|_A$ where $H_{\theta}=\textit{Re}_A(e^{i\theta }T)$. Now,
\begin{eqnarray*}
		H_{\theta} & = & \frac{1}{2}(e^{i\theta} T + e^{-i\theta} T^{\sharp_A}) \\
		\Rightarrow 4 {H_{\theta}}^2  &= & e^{2i\theta} T^2 + e^{-2i\theta} {T^{\sharp_A}}^2 + T^{\sharp_A}T+TT^{\sharp_A} \\ 
		\Rightarrow {H_{\theta}}^2 & = & \frac{1}{2}\textit{Re}_A(e^{2i\theta} T^2) +\frac{1}{4}(T^{\sharp_A}T+TT^{\sharp_A})\\
		\Rightarrow \|{H_{\theta}}^2\|_A & \leq & \frac{1}{2} \big\|\textit{Re}_A(e^{2i\theta} T^2)\big\|_A +\frac{1}{4}\big\|T^{\sharp_A}T+TT^{\sharp_A}\big\|_A 
\end{eqnarray*}
	For $r\geq 1$,  $t^r$ and $t^{\frac{1}{r}}$ are convex and  concave  functions respectively and using that we get,
\begin{eqnarray*}
		\|{H_{\theta}}^2\|_A^r &\leq & \left \{\frac{1}{2} \big\|\textit{Re}_A(e^{2i\theta} T^2)\big\|_A +\frac{1}{2}\left\|\frac{T^{\sharp_A}T+TT^{\sharp_A}}{2}\right\|_A \right\}^r\\
		&\leq & \frac{1}{2} \big\|\textit{Re}_A(e^{2i\theta} T^2)\big\|_A^r +\frac{1}{2}\left\|\frac{T^{\sharp_A}T+TT^{\sharp_A}}{2}\right\|_A^r\\
		&\leq & \frac{1}{2} \big\|\textit{Re}_A(e^{2i\theta} T^2)\big\|_A^r +\frac{1}{2}\left\|\left(\frac{(T^{\sharp_A}T)^r+(TT^{\sharp_A})^r}{2}\right)^{\frac{1}{r}}\right\|_A^r\\
		&= & \frac{1}{2} \big\|\textit{Re}_A(e^{2i\theta} T^2)\big\|_A^r +\frac{1}{2}\left\|\frac{(T^{\sharp_A}T)^r+(TT^{\sharp_A})^r}{2}\right\|_A\\
		&\leq&\frac{1}{2}w^r_A(T^2)+\frac{1}{4}\left\|(T^{\sharp_A}T)^r+(TT^{\sharp_A})^r\right\|_A.
\end{eqnarray*}
Taking supremum over $\theta \in \mathbb{R}$, we get
\begin{eqnarray*}
		w^{2r}_A(T)&\leq&\frac{1}{2}w^r_A(T^2)+\frac{1}{4}\big\|(T^{\sharp_A}T)^r+(TT^{\sharp_A})^r\big\|_A.
\end{eqnarray*}
\end{proof}

\begin{remark}
Here we would like to remark that if we take $r=1$ in the above Theorem \ref{theorem:5}, we get the inequality \cite[Th. 2.11]{Z} proved by Zamani .
\end{remark}

Now we obtain a lower bound for A-numerical radius.

\begin{theorem}\label{theorem:6}
	Let $T \in \mathcal{B}_A(\mathcal{H})$ where $A>0$. Then  
	\[w^{4}_A(T)\geq \frac{1}{4}C^2_A(T^2)+\frac{1}{8}c_A(T^2P+PT^2)+\frac{1}{16}\|P\|_A^2,\]  where $P=T^{\sharp_A}T+TT^{\sharp_A}, C_A(T)=\inf_{\|x\|_A=1}\inf_{\phi \in \mathbb{R}}\|\textit{Re}_A(e^{i\phi} T)x\|_A.$
\end{theorem}

\begin{proof}
	We know that $w_A(T)=\sup_{\phi \in \mathbb{R}}\|H_\phi\|_A$ where $H_{\phi}=\textit{Re}_A(e^{i\phi }T)$. Let $x$ be a unit vector in $H$ and $\theta$ be a real number such that $$e^{2i\theta}\langle(T^2P+PT^2)x,x\rangle_A = |\langle(T^2P+PT^2)x,x\rangle_A|.$$ Then,
	\begin{eqnarray*}
		H_{\theta} & = & \frac{1}{2}(e^{i\theta} T + e^{-i\theta} T^{\sharp_A}) \\
		\Rightarrow 4 {H_{\theta}}^2  &= & e^{2i\theta} T^2 + e^{-2i\theta} {T^{\sharp_A}}^2 + P \\ 
		\Rightarrow 16 {H_{\theta}}^4  &= & \big(e^{2i\theta} T^2 + e^{-2i\theta} {T^{\sharp_A}}^2 + P \big) \big(e^{2i\theta} T^2 + e^{-2i\theta} {T^{\sharp_A}}^2 + P \big)\\
		&= & \big(e^{2i\theta} T^2 + e^{-2i\theta} {T^{\sharp_A}}^2 \big)^2+\big(e^{2i\theta} T^2 + e^{-2i\theta} {T^{\sharp_A}}^2 \big)P \\ 
		&& +P \big(e^{2i\theta} T^2 + e^{-2i\theta} {T^{\sharp_A}}^2 \big)+P^2\\
		& = & 4\big(\textit{Re}_A(e^{2i\theta} T^2)\big)^2+ 2 \textit{Re}_A(e^{2i\theta} (T^2P+PT^2))  +P^2\\
		\Rightarrow 16 w^4_A(T) &\geq&  \|4\big(\textit{Re}_A(e^{2i\theta} T^2)\big)^2+ 2 \textit{Re}_A(e^{2i\theta} (T^2P+PT^2))  +P^2\|_A\\
		&\geq&  | \langle \big(4\big(\textit{Re}_A(e^{2i\theta} T^2)\big)^2+ 2 \textit{Re}_A(e^{2i\theta} (T^2P+PT^2))  +P^2\big)x,x \rangle_A| \\
		&=& | 4 \langle\big(\textit{Re}_A(e^{2i\theta} T^2)\big)^2x,x \rangle_A + 2 \textit{Re}_A(e^{2i\theta} \langle (T^2P+PT^2)x,x\rangle_A)  +\langle P^2x,x \rangle_A | \\
		&=&  4 \|\big(\textit{Re}_A(e^{2i\theta} T^2)\big)x \|_A^2+ 2 |\langle (T^2P+PT^2)x,x\rangle_A|  +\|Px\|_A^2 \\
		&\geq&  4 \|\big(\textit{Re}_A(e^{2i\theta} T^2)\big)x \|_A^2+ 2 c_A(T^2P+PT^2)  +\|Px\|_A^2 \\
		\Rightarrow 16 w^4_A(T) &\geq& 4 C^2_A(T^2)+ 2 c_A(T^2P+PT^2)  +\sup_{\|x\|_A=1}\|Px\|_A^2 \\
		&=&  4 C^2_A(T^2)+ 2 c_A(T^2P+PT^2)  +\|P\|_A^2 \\
		\Rightarrow w^{4}_A(T)&\geq& \frac{1}{4}C^2_A(T^2)+\frac{1}{8}c_A(T^2P+PT^2)+\frac{1}{16}\|P\|_A^2.	
	\end{eqnarray*}
	This completes the proof. 
\end{proof}

\begin{remark}
	It is clear that  $\frac{1}{4}C^2_A(T^2)+\frac{1}{8}c_A(T^2P+PT^2)+\frac{1}{16}\|P\|_A^2\geq  \frac{1}{16}\|T^{\sharp_A}T+TT^{\sharp_A}\|_A^2 \geq \frac{1}{16}\|T\|_A^4.$ So, if $A>0$ then the inequality obtained in Theorem \ref{theorem:6} is better than the first inequality in  \cite [Cor. 2.8]{Z}, obtained by Zamani.
\end{remark}

\section{\textbf{A-numerical radius inequalities for product of operators in $\mathcal{B}_A(\mathcal{H})$}}

We begin this section with the following  $A$-numerical radius inequality for sum of product of operators.

\begin{theorem}\label{theorem:H1}
	Let $ P,Q,X,Y \in \mathcal{B}_A(\mathcal{H})$ where $A>0$. Then
	\[w_A(PXQ^{\sharp_A} \pm QYP^{\sharp_A}) \leq 2\|P\|_A\|Q\|_Aw_B\left(\begin{array}{cc}
	O&X \\
	Y&O
	\end{array}\right).\]
	In particular,  \[w_A(PXQ^{\sharp_A} \pm QXP^{\sharp_A})\leq2\|P\|_A\|Q\|_Aw_A(X).\]
\end{theorem}

\begin{proof}
	Let $C=\left(\begin{array}{cc}
	P&Q \\
	O&O
	\end{array}\right)$ and $Z=\left(\begin{array}{cc}
	O&X \\
	Y&O
	\end{array}\right)$. Then from an easy calculation we get,  \[CZC^{\sharp_B}=\left(\begin{array}{cc}
	PXQ^{\sharp_A}+QYP^{\sharp_A}&O \\
	O&O
	\end{array}\right).\] Therefore,
	\begin{eqnarray*}
		w_A(PXQ^{\sharp_A}+QYP^{\sharp_A})&=&w_B\left(\begin{array}{cc}
			PXQ^{\sharp_A}+QYP^{\sharp_A}&O \\
			O&O
		\end{array}\right) \\
		&=& w_B(CZC^{\sharp_B}),~~ \mbox{using Lemma }\ref{lemma:1}~~ (i) \\
		&\leq& \|C\|_B^2w_B(Z), ~~\mbox{using \cite[Lemma 4.4]{Z}} \\
		&=& \|PP^{\sharp_A}+QQ^{\sharp_A}\|_Aw_B(Z)\\
		&\leq& (\|P\|_A^2 + \|Q\|_A^2)w_B(Z).
	\end{eqnarray*}
	Replacing $P ~~\mbox{and} ~~ Q $ by $tP~~\mbox{and } ~~ \frac{1}{t}Q$ respectively with $t > 0$ in this above inequality,  we get
	\[w_A(PXQ^{\sharp_A}+QYP^{\sharp_A}) \leq \left(\frac{t^4\|P\|_A^2 + \|Q\|_A^2}{t^2}\right)w_B(Z).\]
Note that \[\min_{t>0} \frac{t^4\|P\|_A^2 + \|Q\|_A^2}{t^2}= 2\|P\|_A\|Q\|_A\]
	and so
	\begin{eqnarray*}
		w_A(PXQ^{\sharp_A}+QYP^{\sharp_A}) \leq 2 \|P\|_A\|Q\|_A w_B\left(\begin{array}{cc}
			O&X \\
			Y&O
		\end{array}\right).
	\end{eqnarray*}
	Replacing $Y$ by $-Y$ in the above inequality and using Lemma \ref{lemma:1} (iii), we get
	\begin{eqnarray*}
		w_A(PXQ^{\sharp_A}-QYP^{\sharp_A}) \leq 2 \|P\|_A\|Q\|_A w_B\left(\begin{array}{cc}
			O&X \\
			Y&O
		\end{array}\right).
	\end{eqnarray*}
Taking $X=Y$ and using Lemma \ref{lemma:1} (iv), we get
	\[w_A(PXQ^{\sharp_A} \pm QXP^{\sharp_A}) \leq 2\|P\|_A\|Q\|_Aw_A(X).\]
	This completes the proof of the theorem.
\end{proof}

\begin{remark}
Here we note that the inequality \[w_A(PXQ^{\sharp_A} + QYP^{\sharp_A})\leq 2\|P\|_A\|Q\|_Aw_B\left(\begin{array}{cc}
	O&X \\
	Y&O
	\end{array}\right)\] in Theorem \ref{theorem:H1} holds also when $A\geq 0$. 
\end{remark}

Considering $X=Y=T $ (say), $P=I$ in Theorem \ref{theorem:H1}, we get the following inequality.

\begin{cor}\label{cor2}
	Let $T, Q \in \mathcal{B}_A(\mathcal{H})$ where $A>0$. Then
	\[w_A(TQ^{\sharp_A} \pm QT) \leq 2w_A(T)\|Q\|_A.\]
\end{cor}

Next we prove the following lemma, the idea of which is based on the result \cite[Lemma $3$]{BS} proved by Bernau and Smithes.

\begin{lemma} \label{lemma:G Y and S}
Let $X, T, Y \in \mathcal{B}_A(\mathcal{H})$ where $A>0$.  Then,  for all $x\in \mathcal{H}$
\begin{eqnarray*}\label{number G1}
|\langle X^{\sharp_A}TYx,x\rangle_A|+|\langle Y^{\sharp_A}TXx,x\rangle_A|\leq 2w_A(T)\|Xx\|_A\|Yx\|_A.
\end{eqnarray*}
\end{lemma}

\begin{proof} 
Let $x\in \mathcal{H}$ and ${\theta, \phi}$ be real numbers such that $e^{i\phi}\langle Y^{\sharp_A}TXx,x \rangle_A=|\langle Y^{\sharp_A}TXx,x \rangle_A|$, $e^{2i\theta}\langle e^{-i\phi}X^{\sharp_A}TYx,x \rangle_A=|\langle e^{-i\phi}X^{\sharp_A}TYx,x \rangle_A|=|\langle X^{\sharp_A}TYx,x \rangle_A|.$
Then for non-zero real number $\lambda$, we have

\begin{eqnarray*}
&& 2e^{2i\theta}\langle TYx,e^{i\phi}Xx \rangle_A+2 e^{i\phi}\langle TXx,Yx \rangle_A\\ 
&& \hspace{.8 cm}=\langle e^{i\theta}T\left(\lambda e^{i\theta}Yx+\frac{1}{\lambda}e^{i\phi}Xx \right), \lambda e^{i\theta}Yx+\frac{1}{\lambda}e^{i\phi}Xx \rangle_A\\ 
&& \hspace{1.8 cm} - \langle e^{i\theta}T\left(\lambda e^{i\theta}Yx-\frac{1}{\lambda}e^{i\phi}Xx\right), \lambda e^{i\theta}Yx-\frac{1}{\lambda}e^{i\phi}Xx \rangle_A\\
\Rightarrow && 2e^{2i\theta}\langle e^{-i\phi}X^{\sharp_A}TYx,x \rangle_A+2 e^{i\phi}\langle Y^{\sharp_A}TXx,x \rangle_A\\ 
&& \hspace{.8 cm}=\langle e^{i\theta}T\left(\lambda e^{i\theta}Yx+\frac{1}{\lambda}e^{i\phi}Xx \right), \lambda e^{i\theta}Yx+\frac{1}{\lambda}e^{i\phi}Xx \rangle_A\\ 
&& \hspace{1.8 cm} - \langle e^{i\theta}T\left(\lambda e^{i\theta}Yx-\frac{1}{\lambda}e^{i\phi}Xx\right), \lambda e^{i\theta}Yx-\frac{1}{\lambda}e^{i\phi}Xx \rangle_A\\
\Rightarrow && 2\left|\langle X^{\sharp_A}TYx,x \rangle_A \right|+2 \left |\langle Y^{\sharp_A}TXx,x \rangle_A \right |\\ 
&& \hspace{.8 cm}=\langle e^{i\theta}T\left(\lambda e^{i\theta}Yx+\frac{1}{\lambda}e^{i\phi}Xx \right), \lambda e^{i\theta}Yx+\frac{1}{\lambda}e^{i\phi}Xx \rangle_A\\
&& \hspace{1.8 cm} - \langle e^{i\theta}T\left(\lambda e^{i\theta}Yx-\frac{1}{\lambda}e^{i\phi}Xx\right), \lambda e^{i\theta}Yx-\frac{1}{\lambda}e^{i\phi}Xx \rangle_A
\end{eqnarray*}
\begin{eqnarray*}
\Rightarrow && 2\left|\langle X^{\sharp_A}TYx,x \rangle_A \right|+2 \left |\langle Y^{\sharp_A}TXx,x \rangle_A \right |\\ 
&& \hspace{.8 cm}\leq \left |\langle e^{i\theta}T\left(\lambda e^{i\theta}Yx+\frac{1}{\lambda}e^{i\phi}Xx \right), \lambda e^{i\theta}Yx+\frac{1}{\lambda}e^{i\phi}Xx \rangle_A\right |\\
&& \hspace{1.8 cm} + \left |\langle e^{i\theta}T\left(\lambda e^{i\theta}Yx-\frac{1}{\lambda}e^{i\phi}Xx\right), \lambda e^{i\theta}Yx-\frac{1}{\lambda}e^{i\phi}Xx \rangle_A \right|\\
\Rightarrow && 2\left|\langle X^{\sharp_A}TYx,x \rangle_A \right|+2 \left |\langle Y^{\sharp_A}TXx,x \rangle_A \right |\\  
&&  \hspace{.8 cm}\leq w_A(T) \left( \left\|\lambda e^{i\theta}Yx+\frac{1}{\lambda}e^{i\phi}Xx\right\|_A^2+ \left\|\lambda e^{i\theta}Yx-\frac{1}{\lambda}e^{i\phi}Xx\right\|_A^2 \right)\\
\Rightarrow && \left|\langle X^{\sharp_A}TYx,x \rangle_A \right|+ \left |\langle Y^{\sharp_A}TXx,x \rangle_A \right | \leq w_A(T) \left( \lambda^2 \|Yx\|_A^2+\frac{1}{\lambda^2}\|Xx\|_A^2 \right).
\end{eqnarray*}
This holds for all non-zero real $\lambda.$ If $\|Yx\|_A \neq 0,$ then we choose $\lambda^2=\frac{\|Xx\|_A}{\|Yx\|_A}.$  So, we get
\begin{eqnarray*}
|\langle X^{\sharp_A}TYx,x\rangle_A|+|\langle Y^{\sharp_A}TXx,x\rangle_A|\leq 2w_A(T)\|Xx\|_A\|Yx\|_A.
\end{eqnarray*}
Clearly this inequality also holds when $\|Yx\|_A=0$, i.e., $Yx=0$. This completes the proof of the lemma.
\end{proof}

\begin{remark}
	In \cite{BPN} we have already generalized the result obtained by Bernau and Smithes \cite[Lemma $3$]{BS} and proved some important numerical radius inequalities.
	
\end{remark}

Now using Lemma \ref{lemma:G Y and S}, we obtain the following inequalities involving A-numerical radius, A-Crawford number and A-operator norm.  

\begin{theorem}\label{theorem: G1}
Let $X, T, Y \in \mathcal{B}_A(\mathcal{H})$ where $A>0$. Then 
\begin{eqnarray*}
c_A(X^{\sharp_A}TY)+w_A(Y^{\sharp_A}TX) &\leq & 2w_A(T)\|X\|_A\|Y\|_A,
\end{eqnarray*}
\begin{eqnarray*}
w_A(X^{\sharp_A}TY)+c_A(Y^{\sharp_A}TX) &\leq & 2w_A(T)\|X\|_A\|Y\|_A.
\end{eqnarray*}
\end{theorem}

\begin{proof}
Taking $\|x\|_A=1$ in Lemma \ref{lemma:G Y and S}, we have
\begin{eqnarray*}
|\langle X^{\sharp_A}TYx,x\rangle_A|+|\langle Y^{\sharp_A}TXx,x\rangle_A|&\leq& 2w_A(T)\|X\|_A\|Y\|_A\\
\Rightarrow c_A(X^{\sharp_A}TY)+|\langle Y^{\sharp_A}TXx,x\rangle_A| &\leq& 2w_A(T)\|X\|_A \|Y\|_A.
\end{eqnarray*}
Taking supremum over $\|x\|_A=1$, we get
\[c_A(X^{\sharp_A}TY)+w_A(Y^{\sharp_A}TX) \leq 2w_A(T)\|X\|_A \|Y\|_A. \] 
Again taking $\|x\|_A=1$ in Lemma \ref{lemma:G Y and S},  we have
\begin{eqnarray*}
|\langle X^{\sharp_A}TYx,x\rangle_A|+|\langle Y^{\sharp_A}TXx,x\rangle_A|&\leq& 2w_A(T)\|X\|_A\|Y\|_A\\
\Rightarrow |\langle X^{\sharp_A}TYx,x\rangle_A| + c_A(Y^{\sharp_A}TX) &\leq& 2w_A(T)\|X\|_A \|Y\|_A.
\end{eqnarray*}
Taking supremum over $\|x\|_A=1$, we get
\[w_A(X^{\sharp_A}TY)+c_A(Y^{\sharp_A}TX) \leq 2w_A(T)\|X\|_A \|Y\|_A. \]
This completes the proof of the theorem.\\
\end{proof}

Now taking $Y=I, T=X$ and $X=Y$ in the above Theorem \ref{theorem: G1}, we get the following upper bounds for the numerical radius of product of two operators, which improve on the existing bounds.

\begin{cor}\label{cor: G12}
Let $X,Y \in \mathcal{B}_A(\mathcal{H})$ where $A>0$. Then the following inequalities hold:
\begin{eqnarray*}
w_A(XY) &\leq& 2w_A(X)\|Y\|_A-c_A(Y^{\sharp_A}X), \\
 w_A(XY) &\leq& 2w_A(Y)\|X\|_A-c_A(YX^{\sharp_A}).
\end{eqnarray*}
\end{cor}
 
\begin{remark}
For $A>0$, it is clear that the inequalities obtained in Corollary \ref{cor: G12} improve on the inequalities $w_A(XY)\leq 2w_A(X)\|Y\|_A $ and $w_A(XY)\leq 2w_A(Y)\|X\|_A$,  (see \cite[Th. 3.4]{Z}).\\ 
\end{remark}

Finally using Lemma \ref{lemma:G Y and S} we obtain new inequalities for B-numerical radius of $2 \times 2$ operator matrices with zero operators as main diagonal entries.

\begin{theorem}\label{theorem: G2}
Let $X, Y \in \mathcal{B}_A(\mathcal{H})$ where $A>0$. Then the following inequalities hold:
\begin{eqnarray*}
(i)~~\|X\|_A^2+ c_A(YX) &\leq& 2w_B\left(\begin{array}{cc}
    O&X \\
    Y&O
 \end{array}\right) \|X\|_A,
\end{eqnarray*}
\begin{eqnarray*}
(ii)~~m_A^2(X)+ w_A(YX) &\leq& 2w_B\left(\begin{array}{cc}
    O&X \\
    Y&O
 \end{array}\right) \|X\|_A,
\end{eqnarray*}
\begin{eqnarray*}
(iii)~~\|Y\|_A^2+ c_A(XY) &\leq& 2w_B\left(\begin{array}{cc}
    O&X \\
    Y&O
 \end{array}\right) \|Y\|_A,
\end{eqnarray*}
\begin{eqnarray*}
(iv)~~m_A^2(Y)+ w_A(XY) &\leq& 2w_B\left(\begin{array}{cc}
    O&X \\
    Y&O
 \end{array}\right) \|Y\|_A.
\end{eqnarray*}
\end{theorem}

\begin{proof}
	
Taking  $X=T$ and $Y=I$ in Lemma \ref{lemma:G Y and S} we get,	
\begin{eqnarray*}
	\|Tx\|_A^2+|\langle T^2x,x\rangle_A| &\leq& 2w_A(T)\|Tx\|_A \|x\|_A.
\end{eqnarray*}	
This also holds if we take  $T=\left(\begin{array}{cc}
    O&X \\
    Y&O
\end{array}\right)$ and $x=(x_1,x_2) \in \mathcal{H} \oplus \mathcal{H}$ with $\|x\|_B=1$, i.e., $\|x_1\|_A^2+\|x_2\|_A^2=1$. Therefore we get,
\begin{eqnarray*}
\|Xx_2\|_A^2+\|Yx_1\|_A^2+|\langle XYx_1,x_1\rangle_A+\langle YXx_2,x_2\rangle_A| &\leq &   2w_B (T) \left( \|Xx_2\|_A^2+\|Yx_1\|_A^2\right)^{\frac{1}{2}}.
\end{eqnarray*}
Taking $x_1=0$, we get
\begin{eqnarray*}
\|Xx_2\|_A^2+|\langle YXx_2,x_2\rangle|_A &\leq &   2 w_B \left(\begin{array}{cc}
    O&X \\
    Y&O
\end{array}\right)  \|Xx_2\|_A\\
\Rightarrow \|Xx_2\|_A^2+|\langle YXx_2,x_2\rangle_A| &\leq &   2 w _B\left(\begin{array}{cc}
    O&X \\
    Y&O
\end{array}\right)  \|X\|_A\\
\Rightarrow \|Xx_2\|_A^2+c_A(YX) &\leq &   2 w _B\left(\begin{array}{cc}
    O&X \\
    Y&O
\end{array}\right)  \|X\|_A
\end{eqnarray*}
Taking supremum over $\|x_2\|_A=1$, we get the inequality (i), i.e., 
\begin{eqnarray*}
\|X\|_A^2+c_A(YX) &\leq &   2 w _B\left(\begin{array}{cc}
    O&X \\
    Y&O
\end{array}\right)  \|X\|_A.
\end{eqnarray*}
Again from the inequality 
\begin{eqnarray*}
\|Xx_2\|_A^2+|\langle YXx_2,x_2\rangle_A| &\leq &   2 w_B \left(\begin{array}{cc}
    O&X \\
    Y&O
\end{array}\right)  \|X\|_A, ~~\mbox{we get}
\end{eqnarray*} 
\begin{eqnarray*}
m_A^2(X)+|\langle YXx_2,x_2\rangle_A| &\leq &   2 w_B \left(\begin{array}{cc}
    O&X \\
    Y&O
\end{array}\right)  \|X\|_A.
\end{eqnarray*}
Taking supremum over $\|x_2\|_A=1$, we get the inequality (ii), i.e.,
\begin{eqnarray*}
m_A^2(X) + w_A(YX) &\leq &   2 w_B \left(\begin{array}{cc}
    O&X \\
    Y&O
\end{array}\right) \|X\|_A.
\end{eqnarray*}
Similarly taking $x_2=0$ and supremum over $ \|x_1\|_A=1,$ we can prove the remaining inequalities.\\
\end{proof}

Next taking $X=Y=T$ in  Theorem \ref{theorem: G2} and using  Lemma \ref{lemma:1} (iv), we get the following lower bounds for A-numerical radius.

\begin{theorem}\label{theorem:lower bounds}
Let $T\in \mathcal{B}_A(\mathcal{H})$ with $\|T\|_A \neq 0$ where $A>0$. Then the following inequalities hold:
\begin{eqnarray*}\label{number 2}
w_A(T)&\geq& \frac{\|T\|_A}{2}+\frac{c_A(T^2)}{2\|T\|_A}, 
\end{eqnarray*}
\begin{eqnarray*}\label{number 3}  
w_A(T)&\geq& \frac{m_A^2(T)}{2\|T\|_A}+\frac{w_A(T^2)}{2\|T\|_A}.
\end{eqnarray*}
\end{theorem}

\begin{remark}
Here we note that the two inequalities obtain in Theorem \ref{theorem:lower bounds} are incomparable. So, using these bounds we have a new lower bound 
\[w_A(T)\geq \frac{1}{2\|T\|_A}\max \big\{ \|T\|_A^2+c_A(T^2), m_A^2(T)+w_A(T^2) \big\},\]
where $T\in \mathcal{B}_A(\mathcal{H})$ with $\|T\|_A \neq 0$. It is clear that this inequality improves on the first inequality in \cite[Cor. 2.8]{Z}.
\end{remark}

\noindent \textbf{Acknowledgements:}

First and third author would like to thank UGC, Govt. of India for the financial support in the form of JRF. Prof. Kallol Paul would like to thank RUSA 2.0, Jadavpur University for the partial support.

\bibliographystyle{amsplain}

\end{document}